\documentclass[leqno]{amsart}

\usepackage{amscd,amssymb}
\usepackage{verbatim}
\usepackage[all]{xy}
\usepackage[czech,english]{babel}
\theoremstyle{plain}
\newtheorem{Lemma}{Lemma}[section]
\newtheorem{Th}[Lemma]{Theorem}
\newtheorem{Prop}[Lemma]{Proposition}
\newtheorem{Cor}[Lemma]{Corollary}

\theoremstyle{definition}
\newtheorem{Def}[Lemma]{Definition}
\newtheorem{Notation}[Lemma]{Notation}
\newtheorem{Constr}[Lemma]{Construction}

\theoremstyle{remark}

\newtheorem{Remark}[Lemma]{Remark}

\newenvironment{Proof}{{\sc Proof.}\ }{~\rule{1ex}{1ex}\vspace{0.5truecm}}

\newcommand{\Mod}{\mbox{\rm Mod-}}
\newcommand{\rmod}{\mbox{\rm mod-}}

\newcommand{\Hom}{\operatorname{Hom}}
\newcommand{\End}{\operatorname{End}}
\newcommand{\Ext}{\operatorname{Ext}}

\newcommand{\Ass}[1]{\operatorname{Ass}\,({#1})}

\newcommand{\Spec}[1]{\operatorname{Spec}\,({#1})}
\newcommand{\mSpec}[1]{\operatorname{mSpec}\,({#1})}
\newcommand{\htt}[1]{\operatorname{ht}\,({#1})}
\newcommand{\Img}{\operatorname{Im}}

\newcommand{\Ker}{\operatorname{Ker}}

\newcommand{\Add}{\operatorname{Add}}
\newcommand{\Prod}{\operatorname{Prod}}
\newcommand{\Cog}{\operatorname{Cog}}

\newcommand{\Inj}{\operatorname{Inj}}

\newcommand{\C}{\mathcal{C}}

\newcommand{\I}{\mathcal{I}}

\newcommand{\Y}{\mathcal{Y}}

\newcommand{\Scal}{\ensuremath{\mathcal{S}}}

\newcommand{\p}{\mathfrak{p}}
\newcommand{\q}{\mathfrak{q}}
\newcommand{\m}{\mathfrak{m}}

\newcommand{\lora}{\longrightarrow}
\newcommand{\mapr}[1]{\stackrel{#1}{\longrightarrow}}

\newcommand{\st}{such that }
\newcommand{\wrt}{with respect to }
\newcommand{\ifa}{if and only if }

\newcommand{\+}{\oplus}

\begin{document}

% --- Basic information (title, authors)

\title{Cotilting modules over commutative noetherian rings}

\author{Jan \v{S}\v{t}ov\'{\i}\v{c}ek}
\address[Jan \v{S}\v{t}ov\'{\i}\v{c}ek]{%
Charles University in Prague, Faculty of Mathematics and Physics \\
Department of Algebra \\
Sokolovska 83, 186 75 Praha 8, Czech Republic
}
\email{stovicek@karlin.mff.cuni.cz}

\author{Jan Trlifaj}
\address[Jan Trlifaj]{%
Charles University in Prague, Faculty of Mathematics and Physics \\
Department of Algebra \\
Sokolovska 83, 186 75 Praha 8, Czech Republic
}
\email{trlifaj@karlin.mff.cuni.cz}

\author{Dolors Herbera}
\address[Dolors Herbera]{%
Departament de Matem\`atiques,
Universitat Aut\`onoma de Barcelona,
E-08193 Bellaterra (Barcelona), Spain
}
\email{dolors@mat.uab.cat}

% The abstract, MSC, keywords
\begin{abstract} Recently, tilting and cotilting classes over commutative noetherian rings have been classified in \cite{APST12}. We proceed and, for each $n$-cotilting class $\mathcal C$, construct an $n$-cotilting module inducing $\mathcal C$ by an iteration of injective precovers. A further refinement of the construction yields the unique minimal $n$-cotilting module inducing $\mathcal C$. Finally, we consider localization: a cotilting module is called ample, if all of its localizations are cotilting. We prove that for each $1$-cotilting class, there exists an ample cotilting module inducing it, but give an example of a $2$-cotilting class which fails this property.
\end{abstract}

\date{\today}
\thanks{The first two named authors were supported by grant GA\v{C}R P201/12/G028 from the Czech Science Foundation. 
Part of this paper was written while the third named author was visiting Charles University; she thanks her host for the kind hospitality. She was also partially supported by DGI MICIIN (Spain) throughout the grant MTM2011-28992-C02-01, and by the Comissionat per
Universitats i Recerca de la Generalitat de Catalunya.}

\dedicatory{Dedicated to the memory of Dieter Happel}

\maketitle

% ------------------------------------------------------------------------------
% The document body
\section{Introduction}
\label{sec:intro}

Tilting and cotilting classes have recently been classified for all commutative noetherian rings in terms of increasing sequences of generalization closed subsets of the spectrum \cite{APST12}, or grade consistent functions on the spectrum \cite{DT}. The classification deals first with the dual setting of cotilting classes $\mathcal C$, where these subsets naturally arise as the sets of associated primes of the cosyzygies of the modules in $\mathcal C$. The tilting classes are treated a posteriori, via the Auslander-Bridger transpose. 

This classification does not give any clue for the structure of the corresponding tilting and cotiliting modules. Indeed, tilting and cotilting modules have so far been constructed only in low dimensional cases: for $1$-Gorenstein rings in \cite{TP09}, and for regular local rings of Krull dimension $2$ in \cite{PT11}. Our main result in Theorem \ref{thm:cotilt-constr} below provides a construction of all cotilting modules over commutative noetherian rings using injective precovers of modules. 

For $n = 0$, the $0$-cotilting modules coincide with the injective cogenerators, and the module $\bigoplus_{\m \in \mSpec R} E(R/\m)$ is the minimal one. Our construction shows that the latter fact extends to an arbitrary finite $n$. More precisely, in Theorem \ref{thm:min-cotilt}, we prove the existence, and describe the structure, of the (unique) minimal $n$-cotilting module inducing an $n$-cotilting class.

The localization of any tilting module at a multiplicative subset $S$ of a commutative noetherian ring $R$ always yields a tilting module over the localized ring $R_S$, \cite{AHT06}. The corresponding result clearly fails already for $0$-cotilting modules, but there is always an injective cogenerator $I$ such that for each multiplicative subset $S$, $I_S$ is an injective cogenerator for $\Mod R_S$. We prove the analogous result, i.e., existence of ample cotilting modules, for all $1$-cotilting classes (Theorem \ref{thm:1-cot}). We finish by constructing $2$-cotilting classes $\C$ over complete regular local rings $R$ of Krull dimension $2$ and prime ideals $\p$, such that \emph{no} cotilting module inducing $\C$ localizes at $\p$ to a cotilting $R_\p$-module (Theorem \ref{thm:2-cot}).

% ------------------------------------------------------------------------------
\section{Preliminaries}
\label{sec:prelim}

Unless stated otherwise, $R$ will denote a commutative noetherian ring, $\Mod R$ the category of all (unitary $R$-) modules, and $\rmod R$ its subcategory consisting of all finitely generated modules.

For a module $M$, we denote by $\Add M$ the class of all direct summands of (possibly infinite) direct sums of copies of the module $M$. Similarly, $\Prod M$ denotes the class of all direct summands of direct products of copies of $M$. Further, for $i < \omega$, we denote by $\mho ^{i}M$ the $i$th cosyzygy in the minimal injective coresolution of $M$ (so in particular, $\mho ^{0}M = M$).

First we recall several basic notions and facts from (infinite dimensional) tilting theory.

\begin{Def} \label{def:tilt} A module $T$ is \emph{tilting}, provided that
\begin{itemize}
\item[\rm{(T1)}] $T$ has finite projective dimension.
\item[\rm{(T2)}] $\Ext ^i_R(T,T^{(\kappa)}) = 0$ for all $1 \leq i$ and all cardinals $\kappa$.
\item[\rm{(T3)}] There exist $r < \omega$ and an exact sequence $0 \to R \to T_0  \to \dots \to T_r \to 0$ where $T_0, \dots, T_r \in \Add T$.
\end{itemize}

\noindent The class $T^\perp := \{ M \in \Mod R \mid \Ext ^i_R(T,M) = 0 \textrm{ for each } i \ge 1 \}$ is the \emph{tilting class} induced by $T$. If $T$ has projective dimension $\leq n$, then $T$ is called an \emph{$n$-tilting module}, and $T^\perp$ the \emph{$n$-tilting class} induced by $T$. In this case, condition (T3) holds for $r = n$.

\noindent If $T$ and $T^\prime$ are tilting modules, then $T$ is \emph{equivalent} to $T^\prime$ in case $T^{\perp} = (T^\prime)^{\perp}$, or equivalently $T^\prime \in \Add T$.
\end{Def}

A special feature of the structure theory of tilting modules over commutative noetherian rings is the absence of non-trivial finitely generated examples: A finitely generated module $T$ is tilting, if and only if $T$ is projective (see \cite[Chapter 13]{GT12} for more details on infinite dimensional tilting theory). 

Dually, we define cotilting modules:

\begin{Def} \label{def:cotilt} A module $C$ is \emph{cotilting} provided that
\begin{itemize}
\item[\rm{(C1)}] $C$ has finite injective dimension.
\item[\rm{(C2)}] $\Ext ^i_R(C^{\kappa},C) = 0$ for all $1 \leq i$ and all cardinals $\kappa$.
\item[\rm{(C3)}] There exists $r < \omega$ and an exact sequence $0 \to C_r  \to \dots \to C_0 \to W \to 0$ where $C_0, \dots, C_r \in \Prod C$ and $W$ is an injective cogenerator for $\Mod R$.
\end{itemize}

\noindent The class $^\perp C := \{ M \in \Mod R \mid \Ext ^i_R(M,C) = 0 \textrm{ for each } i \ge 1 \}$ is the \emph{cotilting class} induced by $C$. If $C$ has injective dimension $\leq n$, then $C$ is called an \emph{$n$-cotilting module}, and ${^\perp C}$ the \emph{$n$-cotilting class} induced by $C$. In this case, condition (C3) holds for $r = n$.

\noindent  If $C$ and $C^\prime$ are cotilting modules, then $C$ is \emph{equivalent} to $C^\prime$ provided that $^{\perp}C = {}^\perp C^\prime$, or equivalently $C^\prime \in \Prod C$.

A cotilting module $C$ is called \emph{minimal} provided that $C$ is isomorphic to a direct summand in any cotilting module equivalent to $C$.
\end{Def}

It is easy to see that a module $C$ is $0$-cotilting, if and only if $C$ is an injective cogenerator for $\Mod R$; in this case $C$ is minimal, if and only if $C \cong W_0 := \bigoplus_{\m\in \mSpec R} E(R/\m)$. 

In Section \ref{sec:min-ind}, we will generalize this to an arbitrary $n \geq 0$ by proving that for each $n$-cotilting class there exists a minimal $n$-cotilting module inducing it. While existence of minimal cotilting modules is a non-trivial fact, their uniqueness up to isomorphism follows easily from their pure-injectivity \cite{Sto06} and from a classic result of Bumby \cite{Bum65}; it does not require the noetherian or commutative assumption on $R$: 

\begin{Lemma} \label{lem:bumby} Let $R$ be an arbitrary ring.
\begin{itemize}
\item[(i)] Let $C$ and $D$ be pure-injective modules such that there exist split embeddings $f\colon C \to D$ and $g\colon D \to C$. Then $C \cong D$.
\item[(ii)] Each cotilting module is pure-injective.
\item[(ii)] Minimal cotilting modules are equivalent, if and only if they are isomorphic. 
\end{itemize}
\end{Lemma}
\begin{Proof} (i) By assumption, $C = g(D) \oplus F$ and $D = E \oplus f(C)$ for some submodules $F \subseteq C$ and $E \subseteq D$. Thus 
$D = E \oplus f(C) = E \oplus f(g(D) \oplus F) = E \oplus fg(D) \oplus f(F)$, and $D = E \oplus fg(E) \oplus (fg)^2(D) \oplus fgf(F) \oplus f(F)$. 
Proceeding similarly, we see that $G = fg(E) \oplus (fg)^2(E) \oplus \dots \oplus (fg)^n(E) \oplus \dots$ is a pure submodule in $f(C)$. 
Then $f(C) = PE(G) \oplus H$, where $PE(G)$ denotes the pure-injective hull of $G$ in $f(C)$. Since $f$ and $g$ are monic, $PE(G) \cong E \oplus PE(G)$. 
Thus $D = E \oplus f(C) = E \oplus PE(G) \oplus H \cong PE(G) \oplus H = f(C) \cong C$.             
   
(ii) This has been established in \cite{Sto06}.   

(iii) now follows by parts (i) and (ii).   
\end{Proof}

If $T$ is an $n$-tilting module, then the dual module $T^* = \Hom _R(M,{W_0})$ is an $n$-cotilting module. Moreover, by \cite{APST12}, each cotilting module $C$ is equivalent to a dual of a tilting module (that is, $C$ is of \emph{cofinite type}). In \cite{APST12}, all cotilting classes of modules have been classified in terms of increasing sequences of generalization closed subsets of $\Spec R$, see Theorem \ref{thm:TAMS} below.

\begin{Remark} \label{rem:fail} \rm The result above concerning cofinite type may fail for commutative, but not noetherian rings. For example, if $R$ is any non-strongly discrete valuation domain, then there exist cotilting modules which are not equivalent to duals of the tilting ones, \cite{Baz07}. 
\end{Remark}

For a module $C$ and $i \geq 1$, we define the classes ${}^{\perp _{\ge i}} C$ and ${}^{\perp _n} C$ as follows
$${}^{\perp _{\ge i}} C = \{ M \in \Mod R \mid \Ext ^j_R(M,C) = 0 \textrm{ for each } j \ge i \}$$
and  ${}^{\perp _n} C = \{ M \in \Mod R \mid \Ext ^n_R(M,C) = 0 \}$.  Bazzoni \cite{Baz04} proved that if $C$ is an $n$-cotilting module and $1 \leq i \leq n$, then ${}^{\perp _{\ge i}} C$ is an $(n-i+1)$-cotilting class.

Moreover, for a module $C$ and $1 \le n \leq \infty$, we denote by $\Cog_n C$ the class of all modules $M$ that fit into a long exact sequence 
$0 \to M \to C_0 \to \dots \to C_i \to \dots$ where for each $i < n$, $C_i$ is a product of copies of $C$.  

We will often use the following characterization of $n$-cotilting modules due to Bazzoni \cite{Baz04}:

\begin{Lemma} \label{lem:Bazzoni} Let $C \in \Mod R$ and $1 \leq n < \omega$. Then $C$ is an $n$-cotilting module, if and only if ${}^\perp C =  \Cog_n C$.
\end{Lemma}

We also recall the following well-known fact.          

\begin{Lemma} \label{lem:basic}
\begin{itemize}
\item[(i)] A pure submodule of an injective module $E$ is injective. In particular, any pure quotient of $E$ is a direct summand of $E$.
\item[(ii)] Let $M$ be a direct limit of a directed system $\{I_\alpha, u_{\beta \alpha}\colon I_\alpha \to I_{\beta}\}_{\alpha \le \beta \in \Lambda}$ of
injective modules. Then $M$ is a direct summand in $\oplus_{\alpha \in \Lambda}I_{\alpha}$, in particular, $M$ is injective.
\end{itemize}
\end{Lemma}
\begin{Proof} $(i)$ Since all cyclic modules are finitely presented, the claim easily follows from the Baer Criterion of Injectivity.

$(ii)$ follows by $(i)$ since the canonical presentation of a direct limit as a homomorphic image of the direct sum is a pure epimorphism.
\end{Proof}

The following lemma will be useful for our construction.

\begin{Lemma} \label{lem:perpgen} Let $C$ be a module of injective dimension $n\ge 0$. 
Assume there is an exact sequence
\[0\to X_0\to X_1\to \cdots \to X_n\]
with $X_i\in {}^\perp C$ for any $i\ge 1$, then $X_0\in {}^\perp C$.
\end{Lemma}

\begin{Proof} We prove the statement by induction on $n$. It is obvious for $n=0$ as then ${^\perp C} = \Mod R$.

Assume $n\ge 1$, and that the claim is true for modules of injective
dimension $n-1$. Set $Y=\mathrm{Coker}\, (X_{0}\to X_1)$.

Since $E(C)/C$ has injective dimension $n-1$ and for $n \geq i\ge 2$,
$X_i\in {}^\perp E(C)/C$, we deduce from the inductive hypothesis
that $Y \in {}^\perp E(C)/C$. Therefore, for $i\ge 1$,
\[\Ext_R^i (X_0,C)\cong \Ext_R^{i+1}(Y,
 C)\cong \Ext_R^i(Y, E(C)/C)=0.\]
That is, $X_0\in {}^\perp C$.
\end{Proof}

We also recall a version of the Homotopy Lemma.

\begin{Lemma} \label{lem:homotopy lemma}
Assume we have the following commutative diagram of modules
\[
\begin{CD}
@.E'_1 @>\mu _1 >> E' _2@ >\mu _2>> E'_3  \\ @. @Vf_1VV  @VVf_2V@VVf_3V  \\
0@>>> K @>\varphi _1 >> E @>\varphi _2 >> L
\end{CD}
\]
with $\mu_2\mu_1 = 0$ and exact bottom row. Moreover, assume that there exists $s_3\colon E'_3\to E$ such that 
$\varphi _2 \circ s_3=f_3$. Then there exists $s_2\colon E'_2\to K$ such that $s_2\circ \mu _1 =f_1$ and $f_2=s_3\circ \mu _2 +\varphi _1 \circ s_2$.
\end{Lemma}

\begin{Proof} Since $\varphi _2(f_2-s_3\mu _2)=0$ and $\varphi _1\colon K\to E$ is the kernel of $\varphi _2$, there exists $s_2\colon E'_2\to K$ such that $f_2-s_3\mu _2=\varphi _1\circ s_2$.  This implies that $\varphi _1\circ f_1=f_2\circ \mu _1=\varphi _1\circ s_2\circ \mu _1$. Since $\varphi _1$ is injective, we deduce that $f _1= s_2\circ \mu _1$.
\end{Proof}

We are going to deal with classes of modules that are both pre-covering and pre-enveloping in the sense of our next definition.

\begin{Def} \label{def:pre} Let $\mathcal C$ be a class of modules. A morphism $f \in \Hom _R(C,M)$ with $C \in \mathcal C$ is a \emph{$\mathcal C$-precover} of the module $M$ provided that for each 
morhism $f^\prime \in \Hom _R(C^\prime,M)$ with $C^\prime \in \mathcal C$ there is a $g \in \Hom_R({C^\prime},C)$ such that $f^\prime = fg$. The class $\mathcal C$ is called \emph{precovering} provided that each module $M$ possesses a $\mathcal C$-precover. 

The $\mathcal C$-precover $f$ is called \emph{special} in case $f$ is surjective and Ker$(f) \in \mbox{Ker} \Ext _R^1(\mathcal C,-)$. A $\mathcal C$-precover is called a $\mathcal C$-cover provided that $g$ is an automorphism of $C$ whenever $g \in \Hom_R(C,C)$ is such that $f = fg$. 

The notions of a \emph{(special) $\mathcal C$-preenvelope}, \emph{$\mathcal C$-envelope} and an \emph{enveloping class} are defined dually.   
\end{Def}

Note that if a $\mathcal C$-precover of a module $M$ is surjective, then so are all $\mathcal C$-precovers of $M$, and dually for the injectivity of $\mathcal C$-preenvelopes. 
For basic properties of precovers and preenvelopes, we refer to \cite[Chapters 5 and 6]{EJ11}, or \cite[Chapter 5]{GT12}. 

\begin{Lemma} \label{lem:monoonto}
Let $\C$ be a class of modules closed under direct sums and direct products which is preenveloping and precovering.  Then the following statements are equivalent:
\begin{itemize}
\item[(i)] Each $\C$-preenvelope of any projective module is an injective morphism.
\item[(ii)] Each $\C$-preenvelope of $R$ is an injective morphism.
\item[(iii)] Each $\C$-precover of any injective module is a surjective morphism.
\item[(iv)] Each $\C$-precover of any injective cogenerator is given by a surjective morphism.
\end{itemize}
\end{Lemma}

\begin{Proof} It is clear that $(i)\Rightarrow (ii)$. For the converse, consider an (injective) $\C$-preenvelope $\mu \colon R\to C$. Let $P$ be a projective module, so that $P$ is a direct summand of $R^{(I)}$ for some $I$. Let $\varepsilon \colon P\to R^{(I)}$ denote the inclusion. Since $\C$ is closed under direct sums, $\mu ^{(I)}\colon R^{(I)}\to C^{(I)}$ is a $\C$-preenvelope.  Then $\mu ^{(I)}\circ \varepsilon$ is an injective $\C$-preenvelope of $P$.

Dually, $(iii)$ and $(iv)$ are equivalent.

$(i)\Rightarrow (iii)$. Let $E$ be an injective module, and let $\Phi \colon C\to E$ be a $\C$-precover. Since $R$ is a generator, there exists a set $I$ and a surjective module homomorphism, $g\colon R^{(I)}\to E$. Let $\mu \colon R^{(I)}\to C'$ be a $\C$-preenvelope. Since, by our hypotesis, $\mu$ is injective, the injectivity of the module $E$ implies that there exists $f\colon C'\to E$ such that $f\circ \mu= g$. Since $g$ is surjective, so is $f$. Since $\Phi$ is a $\C$-precover, there exists $s\colon C'\to C$ such that $\Phi \circ s=f$. Since $f$ is surjective, so is $\Phi$.

Dually, $(iii)$ implies $(ii)$. 
\end{Proof}

% ------------------------------------------------------------------------------
\section{Generalization closed subsets of the Zariski spectrum}
\label{sec:gen-cl}

\begin{Def} \label{def:gen-cl}
A subset $Y$ of $\Spec R$ is said to be \emph{generalization closed} if $\q \in Y$ implies $\p\in Y$ for all $\p\subseteq\q \in \Spec R$.

In this case, we let $\I (Y)=\mathrm{Add} \, (\bigoplus _{\p\in Y}E(R/\p))$
\end{Def}

\begin{Prop} \label{prop:properties}
Let $Y \subseteq \Spec R$ be generalization closed. 
\begin{itemize}
\item[(1)] $\I (Y)$ coincides with the class of all injective modules whose associated primes are contained in $Y$.
\item[(2)] $\I (Y)$ is a definable class closed under extensions. In particular, $\I (Y)$ is closed under pure quotients. 
\item[(3)] $\I (Y)$ is both a covering and an enveloping class.
\end{itemize}
\end{Prop}

\begin{Proof} $(1)$. This holds because each injective module is isomorphic to a direct sum of copies of the indecomposable injective modules $E(R/\p)$ for $\p \in \Spec R$, and $\Ass {E(R/\p)} = \{ \p \}$ for each $\p \in \Spec R$.

$(2)$ In view of Lemma~\ref{lem:basic}, in order to prove that $\I (Y)$ is definable, we only have to show that $\I (Y)$ is closed under direct products. Let
$\{I_i\}_{i\in \Lambda}$ be a family of modules in $\I (Y)$. If
$\p\in \Ass {\prod _{i\in \Lambda}I_i}$, then
$\p=\mathrm{ann}_R((m_i)_{i\in \Lambda})=\bigcap _{i\in
\Lambda}\mathrm{ann}_R(m_i)$ for some $(m_i) \in \prod_i I_i$. Consider $i\in \Lambda$ such that $m_i\neq
0$. Since $I_i\in \I (Y)$, $\Ass{m_iR}\subseteq \Ass{I_i} \subseteq Y$.
Therefore, there exists $\q\in Y$ such that $\p\subseteq \q$. Since
$Y$ is generalization closed, we deduce that $\p\in Y$.

Finally, every definable class is closed under pure quotients by~\cite[Theorem 3.4.8]{Prest09}. 

$(3)$. Definable classes are always preenveloping \cite[Proposition
2.8, Theorem 3.3]{RS98}, and since $\I (Y)$ is a class of injective modules
closed under direct summands, it is also enveloping \cite[Proposition 5.11]{GT12}.

By part $(1)$, $\I (Y)$ is the class of all modules isomorphic to direct sums of copies of the indecomposable injective modules $E(R/\p)$ for $\p \in Y$. 
This class is clearly precovering, and since it is closed under direct limits, it is even covering by \cite[Corollary 5.2.7]{EJ11}.
\end{Proof}

\begin{Cor} \label{cor:onto}
Let $Y \subseteq \Spec R$ be generalization closed. The the following statements are equivalent:
\begin{itemize}
\item[(i)]  $\mathrm{Ass}(R)\subseteq Y$;
\item[(ii)] each $\I (Y)$-preenvelope of any projective module is an injective map;
\item[(iii)] each $\I (Y)$-precover of any injective module is surjective.
\end{itemize}
\end{Cor}

\begin{Proof} In view of Proposition~\ref{prop:properties}(3) and Lemma~\ref{lem:monoonto}, we only need to show that a $\I (Y)$-preenvelope of $R$ is injective, if and only if $(i)$ holds. 
But this is clear, since $\mathrm{Ass} (R) = \mathrm{Ass} (E(R))$, and $R$ can be embedded in a module from $\I (Y)$ if and only if $E(R)\in \I (Y)$.
\end{Proof}

\begin{Constr} \label{constr:cotilt-first}
Let $$Y _0\subseteq Y _1 \subseteq \cdots \subseteq Y _i \subseteq \cdots $$ be a fixed sequence of generalization closed subsets of $\Spec R$.

Let $i\ge 0$. For each injective $R$-module $E$, we can construct a complex 
\begin{equation} \label{eqn:coresolutione}                                                                                               
0\lora C\lora E_0\stackrel{\varphi _0}{\lora}E_1 \lora \cdots \lora E_{i-1}\stackrel{\varphi _{i-1}}{\lora} E_i\stackrel{\varphi _{i}}\lora E \lora 0
\end{equation}
with the following properties: $C=\mathrm{Ker} \, \varphi _0$, $\varphi_i$ is an $\I (Y _i)$-precover of $E$, and for each $j<i$ 
there is a factorization of $\varphi_j$
\[
\xymatrix{
E_j\ar[rr]^{\varphi _j}\ar[dr]_{\Phi _j}  & & E_{j+1} \\
& K_{j+1}\ar[ur]_{\nu_{j+1}}\\
}
\]
such that $K_{j+1}\to E_{j+1}$ is the kernel of $\varphi _{j+1}$ and $\Phi _j$ is an $\I (Y_j)$-precover of $K_{j+1}$.
\end{Constr}

In the notation of Construction~\ref{constr:cotilt-first}, we have the following crucial result.

\begin{Th} \label{prop:exact} Assume that $\mathrm{Ass}\, (\mho ^i R) \subseteq Y _i$ for each $i\ge 0$. Then the complex (\ref{eqn:coresolutione}) is exact.
\end{Th}

\begin{Proof}
We fix an injective module $E$ and prove the statement by induction on $i$. If $i=0$, then $\varphi _0$ is surjective by Corollary~\ref{cor:onto}. 

Assume $i>0$. The inductive hypothesis tells us that $\Phi_1,\dots ,\Phi _{i-1}, \varphi_i$ are surjective, so it remains to prove that $\Phi_0$ is surjective. 
Let $F$ be a free module such that there exists an epimorphism $f\colon F\to K_1$. Then we have the commutative diagram given by the solid arrows
\[\xymatrix{
0 \ar[r]&F\ar[r]^{\mu} \ar[d]_{f} &E_0(F)\ar[r]^{\mu _0} \ar[d]_{f_0} \ar@{.>}[ld]_{s'}&E_1(F) \ar[r] \ar[d]_{f_1}\ar@{.>}[ld]_{s_1}&\cdots \ar[r] & E_{i-1}(F)\ar[r]^{\mu _{i-1}} \ar[d]_{f_{i-1}}&E_{i}(F)\ar[d]_{f_{i}}\ar@{.>}[ld]_{s_{i}}\\  
0\ar[r]& K_1 \ar[r] & E_{1} \ar[r]^{\varphi_1} &E_{2} \ar[r]&\cdots\ar[r] & E_{i} \ar[r]^{\varphi _{i}}& E\ar[r]&0  }\]
where the upper complex is part of a minimal injective coresolution of $F$  and the maps $f_0,\dots ,f_i$  are given by the Comparison Theorem, which we apply using the injectivity of the corresponding terms of the bottom row and the exactness of the upper row. In particular, by induction on $j < i$, we obtain the commutative diagrams 
\[\xymatrix{
0 \ar[r]&{\mho ^{j}F}\ar[r] \ar[d]_{\bar f_j} &{E_j(F)}\ar[r] \ar[d]_{f_j} &{\mho ^{j+1}F} \ar[r] \ar[d]_{\bar f_{j+1}} &0\\  
0 \ar[r]& K_{j+1} \ar[r]^{\nu_{j+1}}  &E_{j+1} \ar[r]^{\Phi_{j+1}} &E_{j+2} \ar[r]&\cdots\ar[r] & E_{i} \ar[r]^{\varphi _{i}}& E\ar[r]&0  }\]
where $\bar f_0 = f$. 

By downward induction on $j \leq i$, we will construct the dotted arrows above; they will give a homotopy between the two complexes.

By the hypothesis on $Y _i$, we have $E_{i}(F)\in \I(Y_i)$. Since $\varphi _i$ is a  $\I (Y_i)$-precover, there exists $s_{i}\colon E_{i}(F)\to E_i$ such that $\varphi_i \circ s_{i}=f_{i}$. By Lemma~\ref{lem:homotopy lemma} there exists $s'_{i-1}\colon E_{i-1}(F)\to K_i$ such that
\[ s'_{i-1}\circ \mu_{i-2}=f'_{i-2} \qquad \textrm{and} \qquad f_{i-1} = s_{i}\circ\mu_{i-1} + \nu_{i}\circ s'_{i-1}, \]
where $f'_{i-2} = \Phi_{i-1} \circ f_{i-2}$. Since $E_{i-1}(F)\in \I (Y _{i-1})$  and $\Phi_{i-1}$ is an $\I (Y _{i-1})$-precover of $K_i$, we deduce that there exists $s_{i-1}\colon E_{i-1}(F)\to E_{i-1}$ such that $\Phi _{i-1} \circ s_{i-1}=s'_{i-1}$. Thus 
\[ f_{i-1} = s_{i}\circ\mu_{i-1} + \varphi_{i-1}\circ s_{i-1}. \]
We also have a commutative diagram of solid arrows
\[\xymatrix{ 
& E _{i-3}(F) \ar[r]^{\mu_{i-3}} \ar[d]_{f'_{i-3}} & E _{i-2}(F)\ar[r]^{\mu _{i-2}}\ar[d]_{f_{i-2}} & E_{i-1}(F) \ar[d]^{s'_{i-1}}\ar@{.>}[ld]_{s_{i-1}} \\
0 \ar[r] & K_{i-1} \ar[r] & E _{i-1}\ar[r]^{\Phi _{i-1}}& K_i  }
\]
where $f'_{i-3} = \Phi_{i-2} \circ f_{i-3}$. Now Lemma~\ref{lem:homotopy lemma} allows us to continue the inductive construction of the homotopy.

In the last stage we get a commutative diagram of solid arrows with exact rows 
\[\xymatrix{0\ar[r]
&F\ar[r]^\mu \ar[d]_{f}&  E _{0}(F)\ar[r]^{\mu _{0}}\ar[d]_{f_0} & E_{1}(F)\ar[d]^{s'_1}\ar@{.>}[ld]_{s_{1}} \\  
0\ar[r] &K_{1} \ar[r] &E _{1}\ar[r]^{\Phi _{1}}&K_{2} }
\]
such that $\Phi_{1}s_1=s'_1$. By Lemma~\ref{lem:homotopy lemma}, there exists $s'\colon E_0(F)\to K_{1}$ with $s'\mu = f$. This finishes the proof of the existence of the homotopy.
  
Finally, we observe that since $f$ is surjective, so is $s'$. Since $E_0(F)\in \I (Y_0)$, there exists $s\colon E_0(F)\to E_0$ such that $s'=\Phi_0 \circ s$. Since $s'$ is surjective, so is $\Phi_0$. This finishes the proof.
\end{Proof}

% ------------------------------------------------------------------------------
\section{Constructing the cotilting modules}
\label{sec:constr}

In this section, we consider increasing sequences, $\Y$, of generalization closed subsets of $\Spec R$ 
\[ \Y \, :  \quad Y_{-1} = \emptyset \subseteq Y _0 \subseteq Y _1 \subseteq \cdots \subseteq Y _i \subseteq \cdots \]
such that
\begin{itemize}
\item[(1)] $\bigcup _{i\ge 0} Y_i =\Spec R$, and
\item[(2)] $\Ass{\mho ^i R} \subseteq Y_i$ for each $i\ge 0$.
\end{itemize}

\begin{Notation} \label{nota:cotilt-classes}
For $\Y$ as above, we denote by $\C (\Y)$ the class of all modules $X$ whose minimal injective coresolution is of the form
\[
0\lora X \lora \bigoplus _{\p\in Y_0} E(R/\p)^{(I_{\p,0})} \lora \cdots \lora \bigoplus _{\p\in Y_i} E(R/\p)^{(I_{\p,i})}\lora \cdots
\]
\end{Notation}

In the special case when there is an $n$ such that 
\begin{itemize}
\item[(3)] $Y_{n-1} \subsetneq Y _n=\Spec R$ 
\end{itemize}
we will also use the notation $\C (Y_0, \dots , Y_{n-1})$ for $\C (\Y)$. In particular, $\C (Y_{0},\dots ,Y_{n-1}) = \Mod R$ for $n = 0$.

We recall  the following recent result from \cite{APST12} which is crucial and motivates our work.  

\begin{Th} \label{thm:TAMS} The increasing sequences $\Y$ satisfying (1)--(3) parametrize all $n$-cotilting classes of modules via the assignment $\Y \mapsto \C (\Y)$.
\end{Th}

The problem left open in \cite{APST12} is to construct a cotilting module $C$ such that $\C := \C (\Y)$ is induced by $C$, that is, $\C  = {}^\perp C$. Our main goal here is to solve this problem. 

We start with an instance of Construction~\ref{constr:cotilt-first} for $E = E(R/\p)$:  

\begin{Constr} \label{constr:cotilt}
Let $i\ge 0$, and $\p\in Y_{i+1} \setminus Y_i$. Then we can construct a complex    
\begin{equation} \label{eqn:coresolutioncp}
0\lora C_\p\lora E_0\stackrel{\varphi _0}{\lora}E_1 \lora \cdots \lora E_{i-1}\stackrel{\varphi _{i-1}}{\lora} E_i\stackrel{\varphi _{i}}\lora E(R/\p)\lora 0
\end{equation}
such that $C_\p=\mathrm{Ker} \, \varphi _0$, $\varphi_i$ is an $\I(Y_i)$-cover of $E(R/\p)$, and for each $j<i$ there is a commutative diagram
\[
\xymatrix{
E_j\ar[rr]^{\varphi _j}\ar[dr]_{\Phi _j}  & & E_{j+1} \\
& K_{j+1}\ar[ur]_{\nu_{j+1}}\\
} 
\]
where $\nu_{j+1}$ is the kernel of $\varphi _{j+1}$ and $\Phi _j$ is an $\I (Y_j)$-cover of $K_{j+1}$.
\end{Constr}

By Theorem~\ref{prop:exact} we have

\begin{Prop} \label{prop:exactcotilt}
The complex (\ref{eqn:coresolutioncp}) is exact for all $i\ge 0$ and $\p\in Y_{i+1} \setminus Y_i$, .
\end{Prop}

For $\p \in Y_0$ we define $C_\p:=E(R/\p)$. Finally, we put

\begin{Notation} \label{nota:C(Y)}
\[ C=C(\Y):= \prod _{\p\in \Spec R} C_\p \]
\end{Notation}

\begin{Lemma} \label{lem:injectivesok} 
Let $\p\in \Spec R$ and $0 \leq i < \omega$ be such that $\p\in Y _{i+1}\setminus Y _i$. 
Then $\Ext_R^{j+1} (E,C_\p)=0$ for all $0 \leq j \leq i$ and $E\in \I (Y _j)$.
\end{Lemma}

\begin{Proof}
Let $E\in \I(Y _j)$. We compute $\Ext_R^{j+1} (E,C_\p)$ by applying the
functor $\Hom_R(E,-)$ to the injective coresolution of $C_\p$ given by (\ref{eqn:coresolutioncp}). Since
$\Phi_j\colon E_j\to K_{j+1}$ is an $\I (Y _j)$-cover, $\Hom_R(E, \Phi _j)$ is onto, whence 
\[ \mathrm{Ker}\, (\Hom_R(E, \varphi_{j+1}))= \Img\Hom_R(E, \varphi_{j}).\]
Therefore, $\Ext_R^{j+1} (E,C_\p)=0$ as claimed.
\end{Proof}

Before proceeding, we recall a simple, but important lemma on morphisms between indecomposable injective modules 
(for a proof, see e.g.\ \cite[3.3.8(5)]{EJ11}): 

\begin{Lemma} \label{lem:hom-injectives}
Let $\p,\q \in \Spec R$. Then $\Hom_R\big(E(R/\p), E(R/\q)\big) \ne 0$ \ifa $\p \subseteq \q$.
\end{Lemma}

\begin{Lemma} \label{lem:associatedok}
Let $\p\in \mathrm{Spec}\,(R)$ and $i\ge 0$ be such that $\p\in Y _{i+1}\setminus Y _i$. 
Let $X$ be a module with $\p\in \mathrm{Ass}\, (X)$. Then $\Ext_R^{i+1}(X,C_\p)\neq 0$.

In particular, the injective dimension of $C_\p$ equals $i+1$.
\end{Lemma}

\begin{Proof}
Since $\p \in Y _{i+1}\setminus Y _i$ and $Y_i$ is generalization closed, Lemma \ref{lem:hom-injectives} gives that $\Hom_R(R/\p,E_i) = 0$. 
So $\Ext_R^{i+1}(R/\p,C_\p)\cong \Ext_R^{1}(R/\p,K_i)\cong 
\Hom_R(R/\p,E(R/\p))\cong R/\p$.

The coresolution (\ref{eqn:coresolutioncp}) shows that the injective
dimension of $C_\p$ is at most $i+1$, whence the inclusion $R/\p\to X$
induces a surjective homomorphism 
\[\Ext_R^{i+1}(X,C_\p) \lora \Ext_R^{i+1}(R/\p,C_\p)\cong R/\p.\]
Therefore $\Ext_R^{i+1} (X,C_\p)\neq 0$, and the injective dimension of $C_\p$ equals $i+1$.
\end{Proof}

\begin{Lemma} \label{lem:shifting}
Let $\p\in \Spec R$ and $0 \leq i$ be such that $\p\in Y_{i+1}\setminus Y_i$. 
Assume that $0 \leq j \le i$ and $X\in \C (Y_{0},\dots ,Y_{j-1})$ are such that $\p\in \mathrm{Ass}\, (\mho ^j(X))$. 
Then $\Ext^{i-j+1}_R(X,C_\p)\neq 0$.
\end{Lemma}

\begin{Proof}
By Lemma~\ref{lem:injectivesok}, the hypotheses imply that
\[\Ext^{i-j+1}_R(X,C_\p)\cong \Ext^{i-j+2}_R(\mho ^1
X,C_\p)\cong \dots \cong \Ext^{i+1}_R(\mho ^j X,C_\p),\] 
and $\Ext^{i+1}_R(\mho ^j X,C_\p)\neq 0$ by Lemma~\ref{lem:associatedok}.
\end{Proof}

\begin{Th} \label{thm:perpcomputation} Let $C$ be as in Notation~\ref{nota:C(Y)}. 
Then $\C (\Y) \supseteq {}^\perp C$. 

Moreover, if $Y _n =\Spec R$ for some $n\ge 1$, then $\C = {}^\perp C$. 
\end{Th}

\begin{Proof} Let $X \in {}^\perp C$. Assume there exists $0 \leq j < \omega$ such that $X \in \C (Y_0, \dots, Y_{j-1})$, but $X \notin \C (Y_0, \dots, Y_j)$. Recall that $\C (Y_0, \dots, Y_{j-1}) = \Mod R$ for $j=0$. Then there exists $\p \in \Ass {\mho ^j X}$ such that $\p \notin Y_j$, and $j \leq i < \omega$ such that $\p \in Y_{i +1} \setminus Y_i$. In this setting, Lemma \ref{lem:shifting} gives $\Ext ^{i-j+1}_R(X,C_\p) \neq 0$, a contradiction. This proves that $^\perp C \subseteq \C (\Y)$.

Assume there is an $n\ge 1$ such that $Y _n =\Spec R$. We will use reverse induction on $0 \leq i < n$ to show that 
$\C (Y _i,\dots , Y _{n-1}) \subseteq  {}^{\perp _{\ge i+1}} C$ (for $i = 0$, we will thus obtain the desired inclusion $\C (\Y) \subseteq {}^\perp C$).  

Let $i=n-1$. Since $C_\p$ has injective dimension $< n$ for each $\p \in Y_{n-1}$, Lemma~\ref{lem:injectivesok} yields $\I (Y _{n-1})\subseteq {}^{\perp _n} C$. As the injective dimension of $C$ is $n$, ${}^{\perp _n} C$ is closed by submodules. Hence $\C (Y_{n-1})\subseteq {}^{\perp _n} C$.

Let $0 \leq i<n-1$. We have $X\in \C (Y _i,\dots , Y _{n-1})$, if and only if $E(X)\in \I ( Y _i)$ and $\mho X\in \C (Y _{i+1},\dots, Y_{n-1})\subseteq {}^{\perp _{\ge i+2}} C$. Applying the functor $\Hom_R(-,C)$ to the exact sequence
\[0\lora X\lora E(X)\lora \mho X\lora 0\]
yields, for each $j\ge 1$, the exact sequence
\[ \Ext_R^{i+j}(E(X), C) \lora \Ext_R^{i+j}(X, C) \lora \Ext_R^{i+j+1}(\mho X,C) = 0.\]
Since for each $\p\in Y _{i}$, the injective dimension of $C_\p$ is at most $i$, we get
\[ \Ext_R^{i+j}(E(X), C)\cong \prod _{\p\in \Spec R \setminus Y _i} \Ext_R^{i+j}(E(X),C_\p).\]
If $\p\in \mathrm{Spec} \, (R)\setminus Y _i$, there exists $n > \ell \ge i$ such that $\p\in Y _{\ell +1}\setminus Y _\ell$. By Lemma~\ref{lem:injectivesok}, $\Ext_R^{i+j}(E(X), C_\p)=0$ for any $i+j\le \ell +1$. If $i+j >\ell +1$ then $\Ext_R^{i+j}(E(X),C_\p)=0$ because the injective dimension of $C_\p$ is $\ell+1$.This shows that  $\Ext_R^{i+j}(X,C)=0$ for all $j\ge 1$ and finishes the inductive argument.
\end{Proof}

\begin{Cor} \label{cor:cotilt-pure-inj}
Assume that there is $n\ge 1$ such that $Y_n = \Spec R$. Then:
\begin{itemize}
\item[(i)] for any set $I$ and for any $j\ge 1$, $\Ext_R^j(C^I,C)=0$.
\item[(ii)] $C$ is a pure injective module.
\end{itemize}
\end{Cor}

\begin{Proof} By construction, $C\in \C (\Y)$. Since $\C (\Y)$ is closed under
products, we deduce claim $(i)$ from Theorem~\ref{thm:perpcomputation}.

We also know by (the proof of) \cite[Proposition 3.15]{APST12} that $\C (\Y)$ is a definable class. Therefore $\C (\Y)$ is closed under pure epimorphic images and $C^{I}/C^{(I)}\in \C (\Y)$. It follows that the summation morphism $C^{(I)} \to C$ extends to a morphism $C^I \to C$ for any set $I$, which is equivalent to $C$ being pure injective (see e.g. \cite[Theorem 2.27]{GT12}). This proves claim~$(ii)$.
\end{Proof}

\begin{Th} \label{thm:cotilt-constr}
Assume that there is $n\ge 1$ such that $Y _n =\Spec R$. Then ${}^\perp C= \Cog_n C$. Therefore, $C$ is an $n$-cotilting module inducing the class $\C (Y_0, \dots, Y_{n-1})$.
\end{Th}

\begin{Proof} Let $X\in {}^\perp C=\C (\Y)$. Then $X$ has a
$\mathrm{Prod} (C)$-preenvelope $\varphi \colon X\to C^{I}$ where
$I=\Hom_R(X,C)$. Since $X\in \C (\Y)$, $\mathrm{Ass}\,
(X)\subseteq Y_0$ and since, for any $\p\in Y _0$, $E(R/\p)$ is a
direct summand of $C$, we conclude that $\varphi$ is injective.
Therefore there is a short exact sequence
\[0\to X\stackrel{\varphi}\to C^I\to Y\to 0.\]
Applying the functor $\Hom_R(-,C)$ and using the equality
$\Ext_R^1 (C^I,C)=0$ we obtain the exact sequence
\[
0 \lora \Hom_R(Y,C) \lora \Hom_R(C^I,C) \mapr{\Hom_R(\varphi, C)} \Hom_R(X,C) \lora \Ext_R^1(Y,C) \lora 0.
\]
Since $\Hom_R(\varphi, C)$ is onto, we deduce that
$\Ext_R^1 (Y,C) = 0$. Since $X$ and $C^I$ are in ${}^\perp C$,
by dimension shifting, we deduce that $Y\in {}^\perp C$. From this
we conclude that $X\in\mathrm {Cog}_\infty C \subseteq
\Cog_n C.$

The inclusion $\Cog_n C \subseteq {}^\perp C$ follows by Lemma~\ref{lem:perpgen} and $C$ is $n$-cotilting by Lemma~\ref{lem:Bazzoni}.
\end{Proof}

% ------------------------------------------------------------------------------
\section{Minimality and indecomposable summands}
\label{sec:min-ind}

In this section, we will show that for each cotilting class there is a minimal cotilting module inducing it. By Lemma \ref{lem:bumby}, this cotilting module is unique up to isomorphism. We are now going to describe its structure.  

We will keep the notation of Section \ref{sec:constr}, and use the parametrization of $n$-cotilting classes given by Theorem~\ref{thm:TAMS}.

The first step in our construction of the minimal cotilting module in $\C (Y_0, \dots ,Y_{n-1})$ is the following lemma.

\begin{Lemma} \label{lem:cotilt-resolution}
Let $C \in \Mod R$ be an $n$-cotilting module \st ${}^\perp C = \C(Y_0, \dots , Y_{n-1})$, with the minimal injective coresolution
\begin{equation} \label{eqn:minC}
0\lora C \lora E_0 \mapr{\varphi_{0}} \cdots \lora E_{j-1} \mapr{\varphi_{j-1}} E_j \mapr{\varphi_j} E_{j+1} \lora \cdots \lora E_n \mapr{\varphi_n} 0.
\end{equation}
Then the following hold:
\begin{itemize}
\item[(i)] For each $0 \leq i < n$, the map $\Psi_i\colon E_i \to \mho ^{i+1}C$ induced by (\ref{eqn:minC}) is a special $\I (Y_i)$-precover of $\mho ^{i+1}C$.
\item[(ii)] Let $0 \le j \le n$ and $\Scal \subseteq Y_j \setminus Y_{j-1}$ be a set of primes which are maximal in $Y_j$ with respect to inclusion of prime ideals. Then there is a split embedding $s\colon \bigoplus_{\p \in \Scal} E(R/\p) \to E_j$ \st $\varphi_j s = 0$ (that is, $\bigoplus_{\p \in \mathcal S} E(R/\p)$ is isomorphic to a direct summand in $\mho ^j C$). 
\end{itemize}
\end{Lemma}

\begin{Proof} (i) This is equivalent to proving that $\Ext ^1_R(E(R/\q),\mho ^{i}C) = 0$ for all $\q \in Y_i$, and clearly $\Ext ^1_R(E(R/\q),\mho ^{i}C) \cong \Ext ^{i+1}_R(E(R/\q),C)$. Since $C$ is equivalent to the $n$-cotilting module $C (\Y)$ defined in Notation~\ref{nota:C(Y)} (see Theorem \ref{thm:cotilt-constr}), it remains to prove that $\Ext ^{i+1}_R(E(R/\q),C_\p) = 0$ for each $\p \in \Spec R$. This is clear for $\p \in Y_i$ from Construction~\ref{constr:cotilt} since then the injective dimension of $C_\p$ is at most $i$. Otherwise, there is $i \leq j < n$ such that $\p \in Y_{j+1} \setminus Y_j$, and $\Ext ^{i+1}_R(E(R/\q),C_\p) = 0$ by Lemma \ref{lem:injectivesok}.
    
(ii) Denote for each $\p \in \Scal$ by $k(\p)$ the residue field of $\p$. We claim that for each $\p \in \Scal$ there exists $0 \ne f_\p \in \Hom_R(k(\p),E_j)$ \st $\varphi_j f_\p = 0$. If $j > 0$, it suffices to prove that $\Ext_R^1(k(\p),\mho ^{j-1} C) \cong \Ext_R^j(k(\p),C) \ne 0$. However, by~\cite[Proposition 3.11]{APST12} we know that $k(\p) \in \C(Y_j,\dots,Y_{n-1}) \setminus \C(Y_{j-1},\dots,Y_{n-1})$, and this implies by~\cite[Corollary 3.16]{APST12} that $\Ext_R^i(k(\p),C) = 0$ for all $i>j$, but not for all $i \ge j$. Thus $\Ext_R^j(k(\p),C) \ne 0$, proving the claim if $j > 0$.
If $j = 0$, we even have $k(\p) \in \C(Y_0,\dots,Y_{n-1})$. Hence $k(\p)$ is cogenerated by $C$, which gives a non-zero composition $k(\p) \to C \to E_0$, proving the claim in the remaining case.

Now consider a map $f_\p\colon k(\p) \to E_j$ provided by the claim. Using the structure of injective modules, we can decompose $E_j$ to $E_j = \bigoplus_{\p \in \Scal} E(R/\p)^{(I_\p)} \+ E'$, where $\Ass {E'} \subseteq Y_j \setminus \Scal$. As $V(\Scal) \cap Y_j = \Scal$, also $\Hom_R(k(\p),E') = 0$, so that $\Img f_{\p} \subseteq E(R/\p)^{(I_\p)}$. Since both $k(\p)$ and $E(R/\p)^{(I_\p)}$ are $R_\p$-modules, $f_{\p}$ is an $R_\p$-homomorphism. As $k(\p)$ is simple over $R_\p$, $f_\p$ is an embedding. So the coproduct map $f\colon \bigoplus_{\p \in \Scal} k(\p) \to E_j$ is injective. Clearly also $\varphi_j f = 0$ as we had $\varphi_j f_\p = 0$ for all $\p \in \Scal$.

To finish the proof, we note that $\Ker\varphi_j = \mho^j C$, and also that
\[ \bigoplus_{\p \in \Scal} E(R/\p)/k(\p) \in \C(Y_j,\dots,Y_{n-1}) = {}^\perp(\mho^j C) \]
since $E(R/\p)/k(\p)$ is a semiartinian $R_\p$-module, hence possesses a filtration with composition factors isomorphic to $k(\p)$; see~\cite[Lemma 1.7]{APST12}. Thus $f\colon \bigoplus_{\p \in \Scal} k(\p) \to \Ker\varphi_j$ extends to $s \colon \bigoplus_{\p \in \Scal} E(R/\p) \to \Ker\varphi_j$. To prove that $s$ is an embedding, it suffices to observe that $(\Ker s) \cap \big(\bigoplus_{\p \in \Scal} k(\p)\big) = 0$ since $f$ is an embedding, and that $\bigoplus_{\p \in \Scal} k(\p)$ is an essential submodule of $\bigoplus_{\p \in \Scal} E(R/\p)$. As the domain of $s$ is injective, $s$ necessarily splits.
\end{Proof}

The following notation will be convenient for further steps of our construction.

\begin{Notation} \label{nota:cotilting-prod}
If $\C \subseteq \Mod R$ is a cotilting class, we denote by $\Inj\C$ the class
\[ \Inj\C := \C \cap \C ^\perp = \{ Y \in \C \mid \Ext^1_R(X,Y) = 0 \textrm{ for each } X \in \C \}. \]
Note that if $C$ is a cotilting module \st ${}^\perp C = \C$, then $\Inj\C = \Prod C$.

If $0 \le j \le n$ and $\Scal \subseteq Y_j \setminus Y_{j-1}$, we construct a module $C_\Scal$ similarly as we did for $C_\p$ in Construction~\ref{constr:cotilt}, just starting with $E = \bigoplus_{\p \in \Scal} E(R/\p)$ instead of $E = E(R/\p)$ as the rightmost injective module. That is, we construct an exact sequence

\begin{equation} \label{eqn:grouped}
0\lora C_\Scal \lora E_0 \mapr{\varphi _0} E_1 \lora \cdots \lora E_{j-2} \mapr{\varphi _{j-2}} E_{j-1} \mapr{\varphi_{j-1}} \bigoplus_{\p \in \Scal} E(R/\p) \lora 0
\end{equation}
where $C_\Scal=\Ker \varphi_0$, $\varphi_{j-1}$ is an $\I(Y_{j-1})$-cover of $\bigoplus_{\p \in \Scal} E(R/\p)$, and for each $\ell<j-1$ there is a commutative diagram
\[
\xymatrix{
E_\ell\ar[rr]^{\varphi_\ell}\ar[dr]_{\Phi_\ell}  & & E_{\ell+1} \\
& K_{\ell+1}\ar[ur]_{\nu_{\ell +1}}\\
} 
\]
where $\nu_{\ell+1}$ is the kernel of $\varphi_{\ell+1}$, and $\Phi_\ell$ is an $\I (Y _\ell)$-cover of $K_{\ell+1}$. Clearly $C_{\{\p\}} = C_\p$ for a single $\p \in Y_j \setminus Y_{j-1}$.
\end{Notation}

Now we can construct the minimal cotilting module (see Definition \ref{def:cotilt}):

\begin{Th} \label{thm:min-cotilt}
Let $\Y$ be a chain of generalization closed subsets of $\Spec R$ satisfying (1), (2) and (3) from Section \ref{sec:constr} (so that $\C (Y_0,\dots, Y_{n-1})$ is an $n$-cotilting class in $\Mod R$). Then there is a minimal $n$-cotilting module $C \in Mod R$ \st ${}^\perp C = \C(Y_0, \dots, Y_{n-1})$. In fact, up to isomorphism
\[ C = C_{\Scal_0} \+ C_{\Scal_1} \+ \cdots \+ C_{\Scal_n}, \]
where $\Scal_j \subseteq Y_j \setminus Y_{j-1}$ is the set of all primes maximal \wrt inclusion in $Y_j \setminus Y_{j-1}$ and the $C_{\Scal_j}$ are as in Notation~\ref{nota:cotilting-prod}.
\end{Th}

\begin{Proof}
First, a straightforward modification of the proof of Theorems~\ref{thm:perpcomputation} and~\ref{thm:cotilt-constr} shows that $C$ is a cotilting module and ${}^\perp C = \C(Y_0, \dots, Y_{n-1})$.

Suppose that $D$ is another cotilting module inducing the cotilting class $\C(Y_0, \dots, Y_{n-1})$. Consider a minimal injective coresolution of $D$,
\begin{equation} \label{eqn:alternative-cotilt}
0\lora D \lora E_0 \mapr{\psi_0} E_1 \lora \cdots \lora E_{n-2} \mapr{\psi_{n-2}} E_{n-1} \mapr{\psi_{n-1}} E_n \lora 0,
\end{equation}
and for each $0 \le j \le n$, the injective coresolution
\[
0\lora C_{\Scal_j} \lora E_{0,j} \mapr{\varphi _0} E_{1,j} \lora \cdots \lora E_{j-2,j} \mapr{\varphi _{j-2}} E_{j,j-1} \mapr{\varphi_{j-1}} E_{j,j} \lora 0
\]
from Notation~\ref{nota:cotilting-prod}, where $E_{j,j} := \bigoplus_{\p \in \Scal_j} E(R/\p)$. We denote the cosyzygies \wrt these injective coresolutions by $L_i = \mho^iD$ and $K_{i,j} = \mho^i C_{\Scal_j}$. In particular, $L_n = E_n$ and $K_{j,j} = \bigoplus_{\p \in \Scal_j} E(R/\p)$.

We will prove by reverse induction on $i = n, \dots, 0$ that $\bigoplus_{i \le j \le n} K_{i,j}$ split embeds into $L_i$. For $i=0$, we will thus obtain our theorem. 

For $i=n$ we know that $\bigoplus_{\p \in \Scal_j} E(R/\p)$ split embeds into $E_n$ by Lemma~\ref{lem:cotilt-resolution}(ii).

Suppose now that $0 \le i < n$. Since $K_{i,i}$ is injective, Lemma~\ref{lem:cotilt-resolution}(ii) even yields a decomposition $E_i = A_i \oplus B_i$ where 
$K_{i,i} \cong A_i \subseteq \Ker \psi _i$. By Lemma~\ref{lem:cotilt-resolution}(i), the morphism $\Psi_i\colon E_i \to L_{i+1}$ induced by (\ref{eqn:alternative-cotilt}) is a special $\I(Y_i)$-precover. Consider its restriction $\Xi _i : B_i \to L_{i+1}$. Since $\Ker \Xi_i = \Ker \Psi_i \cap B_i$ is a direct summand in $\Ker \Psi _i$, also $\Xi _i : B_i \to L_{i+1}$ is a a special $\I(Y_i)$-precover of $L_{i+1}$. 

By the inductive premise, $L_{i+1}$ has a decomposition $L_{i+1} = G_{i+1} \oplus H_{i+1}$ where $G_{i+1}$ is isomorphic to $\bigoplus_{i+1 \le j \le n} K_{i+1,j}$. In particular, the $\I (Y_i)$-cover $f_i$ of $L_{i+1}$ is a direct sum of the $\I (Y_i)$-covers $g_i$ and $h_i$ of $G_{i+1}$ and $H_{i+1}$, respectively. Then $\Ker g_i \cong \bigoplus_{i+1 \le j \le n} K_{i,j}$ is a direct summand in $\Ker f_i$. 

Finally, being a $\I (Y_i)$-cover, $f_i$ is a direct summand in the $\I (Y_i)$-precover $\Xi _i$, \cite[5.1.2]{EJ11}. We can thus conclude that $L_i = \Ker \Psi_i = \Ker \Xi _i \oplus A_i$ has a direct summand $\Ker g_i \oplus A_i$ which is isomorphic to $\bigoplus_{i \le j \le n} K_{i,j}$.                    
\end{Proof}

A very similar argument allows us to classify the indecomposable modules in the class $\Inj\C(Y_0, \dots, Y_{n-1})$. For $n=0$ this just gives the well known parametrization of indecomposable injective modules. In the notation from Construction~\ref{constr:cotilt} and Notation~\ref{nota:cotilting-prod}, we have

\begin{Th} \label{thm:indecomp-cotilt}
If $X \in \Inj\C(Y_0, \dots, Y_{n-1})$ is a non-zero module, then $C_\p$ split embeds into $X$ for some $\p \in \Spec R$. In particular, the indecomposable modules in $\Inj\C(Y_0, \dots, Y_{n-1})$ are parametrized by $\Spec R$.
\end{Th}

\begin{Proof}
Consider a minimal injective coresolution
\[
0\lora X \lora E_0 \mapr{\varphi_0} E_1 \lora \cdots \lora E_{j-2} \mapr{\varphi_{j-2}} E_{j-1} \mapr{\varphi_{j-1}} E_j \lora 0,
\]
of $X$, so that $E_j \ne 0$. Fix a prime $\p$ \st $E(R/\p)$ is a summand of $E_j$. If $j=0$, then the conclusion is clear. Hence assume that $j \ge 1$.

We observe that $\p \not\in Y_{j-1}$. Indeed, if $\p \in Y_{j-1}$ then $E(R/\p) \in \C(Y_{j-1}, \dots, Y_{n-1})$ and the split inclusion $E(R/\p) \to E_j$ would factor through $\varphi_{j-1}$ since $\Ext_R^j(E(R/\p),X) = 0$, contradicting the minimality of the coresolution of $X$.

Now a similar induction as in the proof of Theorem~\ref{thm:min-cotilt} shows that $C_\p$ is a summand of $X$, which implies $C_\p \cong X$ if $X$ is indecomposable.
\end{Proof}

% ------------------------------------------------------------------------------
\section{Ampleness and localization}
\label{sec:local}

If $T$ is a tilting module and $S$ is a multiplicative subset in $R$, then the localization $T_S$ is well-known to be a tilting $R_S$-module (see \cite{AHT06} or \cite[\S 13.3]{GT12}). In particular, the localization of $T$ at any prime ideal $\p$ is a tilting $R_\p$-module. 

However, being a tilting module is not a local property in the sense of \cite{AM69}, that is, $T$ need not be tilting even if $T_\p$ is a tilting $R_\p$-module for each prime ideal $\p \in \Spec R$. For example, let $T$ be the subgroup of $\mathbb Q$ containing $\mathbb Z$ such that $T/\mathbb Z \cong \bigoplus_{p} \mathbb Z/(p)$. Then $T_{(p)}$ is a non-zero free $\mathbb Z _{(p)}$-module for each prime $p$, but $T$ is not a tilting $\mathbb Z$-module, because it is flat, but not projective.

Although in our setting of commutative noetherian rings, each cotilting module is equivalent to the dual of a tilting one, localization does not preserve cotilting modules in general. Already in the case of $0$-cotilting modules (= injective cogenerators), the minimal injective cogenerator $\oplus_{\m \in \mSpec R} E(R/\m)$ localizes to $0$ at each non-maximal prime ideal. However, $\oplus_{\p \in \Spec R} E(R/\p)$ always localizes to an injective cogenerator. This leads to the  following notion:

\begin{Def} \label{def:max} A cotilting module $C$ is \emph{ample} provided that for each multiplicative subset $S$ of $R$, the localized module $C_S$ is a cotilting $R_S$-module. 
\end{Def} 

In this section, we will prove that each $1$-cotilting class is induced by an ample cotilting module, but there are $2$-cotilting classes which fail this property.

We will need the classic fact due to Matlis showing that in our setting, localizations of injective modules are injective (see e.g.\ \cite[3.3.8(6)]{EJ11}):

\begin{Lemma} \label{lem:matlis}
Let $\p \in \Spec R$ and $S$  a multiplicative subset of $R$. Then $E(R/\p) _S = 0$ in case $\p \cap S \neq \emptyset$. If $\p \cap S = \emptyset$, then $E(R/\p) _S = E_{R_S}(R_S/\p_S)$ as $R_S$-modules, and $E(R/\p) _S = E(R/\p)$ as $R$-modules.    
\end{Lemma} 

For the following result, recall that $1$-cotilting classes are parametrized by generalization closed subsets $Y \subseteq \Spec R$ such that $\Ass R \subseteq Y$, see \cite{APST12} or Theorem~\ref{thm:TAMS}. For a multiplicative subset $S$ of $R$, we we will use the notation $Y_S := \{ \p_S \mid \p \in Y \mbox{ and } \p \cap S = \emptyset \}$. Notice that $Y_S$ is generalization closed, and $\Ass {R_S} \subseteq Y_S \subseteq \Spec {R _S}$, so $\C (Y_S)$ is a $1$-cotilting class in $\Mod R_S$, for each multiplicative subset $S$ of $R$. 

\begin{Th} \label{thm:1-cot} Let $\mathcal C$ be a $1$-cotilting class, so $\C = \C (Y)$ where $\Ass R \subseteq Y \subseteq \Spec R$ and $Y$ is closed under generalization. 

\begin{itemize}
\item[(i)] Let $D$ be an arbitrary cotilting module inducing $\C$, and $S$ be a multiplicative subset of $R$. Then
\[ \Cog D_S \subseteq \C (Y_S) = \{ M \in \Mod R_S \mid \Ass M \subseteq Y_S \} \subseteq{}^{\perp} D_S. \]
In particular, if $D_S$ is a cotilting module, then $D_S$ induces the cotilting class $\C (Y_S)$. 
\item[(ii)] There exists an ample $1$-cotilting module $C$ inducing $\C (Y)$.
\end{itemize}
\end{Th}

\begin{proof} (i) Since $\Ass D \subseteq Y$, we have $\Ass {D_S} \subseteq Y_S$. So the $1$-cotilting class $\C (Y_S)$ contains $\mbox{Cog }D_S$.

Let $0 \to D \to A \overset{\varphi}\to B \to 0$ be the minimal injective coresolution of $D$ in $\Mod R$. The $R_S$-module $D_S$ has injective dimension $\leq 1$, so for the inclusion $\C (Y_S) \subseteq {}^{\perp} D_S$, it suffices to prove that $\Ext ^1_{R_S}(E_{R_S}(R_S/\p_S),{D_S}) = 0$ for all $\p \in Y_S$, or the equivalent claim that $\varphi_S$ is a $\I (Y_S)$-precover of $B_S$. 

However, $\varphi$ is a (special) $\I (Y)$-precover of $B$ by Lemma \ref{lem:cotilt-resolution}(i). Let $\p \in Y_S$ and consider $\psi \in \Hom _{R_S}({E_{R_S}(R_S/\p_S)},{B_S})$. By Lemma \ref{lem:matlis}, as $R$-module, $B_S$ is a direct summand in $B$. Let $\psi ^\prime$ denote $\psi$, but viewed as an $R$-homomorphism from $E(R/\p) = E_{R_S}(R_S/\p_S)$ to $B$. Since $\varphi$ is a $\I (Y)$-precover of $B$, $\psi ^\prime$ factors through $\varphi$. That is, there exists $\xi \in 
\Hom _R({E(R/\p)},A)$ such that $\varphi \xi = \psi$. Localizing at $S$, we get $\varphi_S (\xi \otimes _R R_S) = \psi^\prime \otimes _R R_S = \psi$. This proves our claim. 

If $D_S$ is a cotilting $R_S$-module, then $\mbox{Cog }D_S = {}^\perp D_S$, so $D_S$ induces $\C (Y_S)$. 

(ii) If $Y = \Spec R$, then $\mathcal C = \Mod R$; in view of Lemma \ref{lem:matlis}, it suffices to take $C = \bigoplus_{\p \in \Spec R} E(R/\p)$.

Assume $Y \subsetneq \Spec R$. Let $B = \bigoplus_{\q \in \Spec R \setminus Y} E(R/\q)$ and consider the short exact sequence 
$0 \to C_1 \to A \overset{\varphi}\to B \to 0$ where $\varphi$ is the $\I (Y)$-cover of $B$. Let $C_0 = \bigoplus_{\p \in Y} E(R/\p)$.
As in Section \ref{sec:constr}, we see that $C = C_0 \oplus C_1$ is a $1$-cotilting module inducing the class $\C$.   

Let $S$ be any multiplicative subset of $R$. In view of part (i), it remains only to prove that ${}^{\perp} C_S \subseteq \mbox{Cog }C_S$. 
Let $M \in {}^{\perp} C_S$. If $\q_S \in \Ass {M} \setminus Y_S$, then there is a monomorphism $\nu: R_S/\q_S \to B_S$. Since $Y_S$ is closed under generalization, $\nu$ does not factorize through $\varphi_S$ by Lemma \ref{lem:hom-injectives}. Hence $\Ext ^1_{R_S}(R_S/\q_S,{C_S}) \neq 0$ and, since $R_S/\q_S\hookrightarrow M$,  also $\Ext ^1_{R_S}(M,{C_S}) \ne 0$, a contradiction. This proves that $M \in \C (Y_S) \subseteq \Cog ((C_0)_S) = \mbox{Cog }C_S$.  
\end{proof}

\begin{Remark} \label{rem:minimal-ample}
It is not difficult to observe that $C$ as in the proof of Theorem~\ref{thm:1-cot} is in fact a minimal ample cotilting module for $\C$. That is, if $D$ is any other ample cotilting module for $\C$, then $C$ is isomorphic to a direct summand of $D$. Again, a minimal ample cotilting module for $\C$ is unique up to isomorphism by Lemma \ref{lem:bumby}.
\end{Remark}

\medskip
We will now show that Theorem \ref{thm:1-cot} cannot be extended to $2$-cotilting classes. To this purpose, assume that \emph{$R$ is a complete regular local ring $R$ of Krull dimension $2$}. Note that $R$ is a unique factorization domain. 

We will construct a $2$-cotilting class $\C \subseteq \Mod R$ which is not induced by any ample cotilting module. In fact, we will prove a stronger claim: If $D$ is any cotilting module inducing $\C$, then its localization $D_{\p}$ at any $\p \in \Spec R$ of height $1$ is not a cotilting module in $\Mod{R_\p}$. 

We know that $\C$ is of the form $\C = \C(Y_0, Y_1)$ with $Y_0 \subseteq Y_1$ generalization closed subsets of $\Spec R$ \st $Y_i$ contains all primes of height $i$ for $i = 0, 1$. We make the following particular choice:
\[ Y_0 = Y_1 = \Spec R \setminus \{\m\}, \]
where $\m \in \Spec R$ is the maximal ideal. 

First we collect some information about the minimal cotilting module $C$ inducing our particular $\C$.

\begin{Lemma} \label{lem:primes-small-ht}
Let $\p \in \Spec R$ be a prime of height at most $1$. Then $\Ext^2_R\big(E(R/\p),R\big) = 0$.
\end{Lemma}

\begin{Proof}
We know that either $\p = 0$ or $\p$ is generated by an irreducible element of $R$. In either case the projective dimension of $R/\p$ is at most $1$ and $\Ext^2_R(R/\p,R) = 0$. Since $R \cong \End_R\big(E(R/\m)\big)$ is pure-injective and $k(\p) = R_\p \otimes_R R/\p$ is a direct limit of copies of $R/\p$, it follows from~\cite[Lemma 6.28]{GT12} that $\Ext^2_R\big(k(\p),R\big) = 0$. Finally, $E(R/\p)$ is $k(\p)$-filtered and hence $\Ext^2_R\big(E(R/\p),R\big) = 0$.
\end{Proof}

\begin{Cor} \label{cor:resolution-of-R}
Let
\[ 0 \lora R \lora Q \to \bigoplus_{\htt \p = 1} E(R/\p) \mapr{\varphi} E(R/\m) \lora 0 \]
be a minimal injective coresolution of $R$. Then $\varphi$ is an $\I(Y_1)$-cover of $E(R/\m)$.
\end{Cor}

\begin{Proof}
The (special) precovering property was proved in Lemma \ref{lem:primes-small-ht}. Moreover, $\I(Y_1)$ is a covering class by Proposition \ref{prop:properties}(3), so the $\I(Y_1)$-cover $\psi$ of $E(R/\m)$ is a direct summand in $\varphi$ \cite[5.1.2]{EJ11}. By Lemma \ref{lem:hom-injectives}, $E(R/\p)$ must be a direct summand of the domain of $\psi$ for each $\p$ of height $1$, whence $\psi = \varphi$.  
\end{Proof}

Thus, by the construction of the minimal cotilting module for $\C$, $C$ must contain a direct summand $C'$ \st there is a short exact sequence
\[ 0 \lora C' \lora Q^{(I_0)} \oplus \bigoplus_{\htt\p = 1} E(R/\p)^{(I_\p)} \mapr{\vartheta} Q/R \lora 0, \]
where $\vartheta$ is an $\I(Y_0)$-cover.

Thus $C'$ has a minimal injective coresolution of the form
\[ 0 \lora C' \lora Q^{(I_0)} \oplus \bigoplus_{\htt\p = 1} E(R/\p)^{(I_\p)} \lora \bigoplus_{\htt \p = 1} E(R/\p) \mapr{\varphi} E(R/\m) \lora 0 \]
and its localization $C'_\p$ at any prime ideal $\p$ of height $1$ has a minimal injective coresolution of the form
\[ 0 \lora C'_{\p} \lora Q^{(I_0)} \oplus E(R/\p)^{(I_\p)} \lora E(R/\p) \lora 0 \]
In particular, $C'_{\p}$ is not injective.

\begin{Th} \label{thm:2-cot} There is no ample cotilting module inducing the class $\C$. Moreover, if $D$ is any cotilting module inducing $\C$ and $\p$ any prime ideal of height $1$, then $D_{\p}$ is not a cotilting $R_{\p}$-module.
\end{Th}

\begin{Proof}
Suppose for a contradiction that $D$ is an ample $2$-cotilting module inducing the class $\C$. In particular, assume $D_{\p}$ is a cotilting $R_{\p}$-module for any fixed prime of height $1$. Since $R_{\p}$ is a discrete valuation domain, there are only two equivalence classes of cotilting modules: the injective cogenerators and the flat cotilting modules. Since $C^\prime_{\p}$ is a direct summand of $D_{\p}$, the first option does not occur by the dicussion above.

However, $D_{\p}$ cannot be a flat (or equivalently torsion--free) $R_{\p}$-module either. Indeed, Lemma \ref{lem:cotilt-resolution}(ii) implies that $E(R/\p)$ is a direct summand in $D_{\p}$. Thus, $D_{\p}$ is not cotilting in $\Mod{R_{\p}}$.
\end{Proof}

% ----------------------------------------------------------------------------
% The bibliography
\bibliographystyle{abbrv}
\bibliography{ncotiltingnoetherian2}

\end{document}